\documentclass[letterpaper,11pt]{amsart}

\usepackage[left=1in,right=1in,top=.75in,bottom=.75in]{geometry}

\usepackage{latexsym,amssymb,amsmath,amsthm,amsopn}

\newtheorem{theorem}{Theorem}[section]
\newtheorem{maintheorem}{Theorem}

\newtheorem{coro}{Corollary}
\newtheorem{prop}[theorem]{Proposition}
\newtheorem{lemma}[theorem]{Lemma}

\theoremstyle{definition}

\theoremstyle{definition}

\theoremstyle{remark}
\newtheorem{remark}[theorem]{Remark}

\theoremstyle{definition}

\newcommand{\R}{\mathbb{R}}
\newcommand{\RP}{\mathbb{R}\textrm{P}}
\newcommand{\ft}{\mathfrak{t}}
\newcommand{\T}{\mathbb{T}}
\newcommand{\tQ}{\widetilde{Q}}
\newcommand{\C}{\mathbb{C}}
\newcommand{\Z}{\mathbb{Z}}

\begin{document}

\title{Toric Integrable Geodesic Flows in Odd Dimensions}
\author{Christopher R. Lee}
\address{Department of Mathematics, University of Portland, 5000 N. Willamette Blvd., Portland, OR 97203}
\email{leec@up.edu}
\author{Susan Tolman}
\address{Department of Mathematics, University of Illinois, 1409 W. Green St., Urbana, IL 61801}
\email{stolman@math.uiuc.edu}
\date{\today}
%\subjclass[2000]{Primary 53D25; Secondary 53D10}

\begin{abstract}
Let $Q$ be a compact, connected $n$-dimensional Riemannian manifold,
and assume that the geodesic flow  is toric integrable.
If $n \neq 3$ is odd, or if $\pi_1(Q)$ is infinite, we show that the cosphere
bundle 
of $Q$ is equivariantly contactomorphic to the cosphere bundle of the torus $\T^n$.
As a consequence, $Q$ is homeomorphic to $\T^n$.
\end{abstract} 

\maketitle
\section{Introduction}

Let $Q$ be an $n$-dimensional compact Riemannian manifold. 
On the punctured cotangent bundle, 
$T^*Q\smallsetminus{Q}$,
consider the function
\begin{displaymath}
h(q,p) = \sqrt{g^*_{q}(p,p)},
\end{displaymath}
where $g^*$ denotes the dual of the Riemannian metric $g$. 
(Here, $(q,p)$ are local coordinates on $T^*Q$,  with $q \in Q$ and $p \in T_{q}^*Q$.) 
The {\bf geodesic flow} on $T^*Q\smallsetminus{Q}$ is  the flow of the Hamiltonian vector field of 
$h$ with 
respect to the 
standard symplectic form $\Omega = \sum dq_i \wedge dp_{i}$.
The geodesic flow is {\bf toric integrable} if there is an effective  
action of the torus $\T^n = \R^n/\Z^n$ on $T^*Q\smallsetminus{Q}$ that commutes with dilations 
$(q,p) \mapsto (q,e^{t}p)$, preserves the symplectic form $\Omega$, and
commutes with the geodesic flow 
(or, equivalently, preserves $h$). 
For example,
the flat metric on the torus $\T^n$ and the
round metrics on the sphere $S^2$, the projective plane
$\RP^2$, and the lens spaces $S^3/\Z_\ell$ are 
all
toric integrable.
The $\T^n$ action on $T^* \T^n \smallsetminus \T^n$ is 
simply the lift of the natural action on $\T^n$.
In the remaining cases,
the geodesic flow induces a circle action on $T^*Q \smallsetminus Q$
which commutes with the lift of 
the natural $\T^{n-1}$ action on $Q$.

In \cite{tz}, Toth and Zelditch employ toric integrable geodesic flows to examine connections between the 
dynamics of the geodesic flow and 
the eigenfunctions of the Laplace operator. 
Motivated by this work,
Lerman and Shirokova  prove a conjecture of Toth and Zelditch:
every toric integrable metric on a torus is flat \cite{ls}.
(In contrast, 
not every toric integrable metric on $S^2$ is round \cite{cdv}.)
In \cite{l2}, Lerman extends this work to show that the only toric integrable actions on
the punctured cotangent  bundle of the $n$-torus and the $2$-sphere are the standard actions.
In this paper, he poses the following question: are the examples listed above
the only manifolds that admit toric integrable geodesic flows?
The first author partially answered this question  
by proving that  if $Q$ is a compact $3$-dimensional manifold which admits
a toric integrable geodesic flow, then the cosphere bundle $S(T^*Q)$ is either
equivariantly contactomorphic\footnote{\label{Chris}
Theorem 1.1 in \cite{cl} only explicitly states that these manifolds
are diffeomorphic, but clearly proves this stronger claim.} to $\T^3 \times S^2$, 
or is homotopy equivalent to $(S^3/\Z_\ell) \times S^2$, where $\ell \geq 1$ 
\cite[Theorem 1.1]{cl}.
The main goal of this paper is to determine which odd-dimensional compact manifolds
admit toric integrable geodesic flows.

Our approach to toric integrable geodesic flows will be from the perspective of contact geometry.
A {\bf contact form} 
on a $(2n-1)$-dimensional manifold $M$ is a one-form $\alpha$ such that $\alpha \wedge (d\alpha)^{n-1}$ 
is nowhere zero. A (co-orientable) {\bf contact structure} on $M$ is a codimension one subbundle $\xi$ of $TM$
such that $\xi= \ker{\alpha}$ for some contact form $\alpha$. 
A {\bf contact toric manifold} is a $(2n-1)$-dimensional manifold $M$, 
a contact structure $\xi$  on $M$,
and an effective action of the $n$-torus that preserves $\xi$.

For any Riemannian manifold $Q$, the 
restriction of the  Liouville one-form $\alpha = \sum p_i dq_i$ 
to the cosphere bundle $S(T^*Q)$ is a contact form. 
If $Q$ admits a toric 
integrable geodesic flow, then the $\T^n$-action on $T^*Q\smallsetminus{Q}$ preserves both 
$S(T^*Q)$ and $\alpha$; hence,  its cosphere bundle is a contact toric 
manifold. 
By exploiting the classification of 
contact toric manifolds due to Lerman, we are able to 
show that  many contact toric manifolds  cannot arise 
as cosphere bundles -- they have the wrong 
cohomology.

\begin{maintheorem}\label{mainresult}
Let $Q$ be a compact, connected 
$n$-dimensional Riemannian manifold, and assume that the geodesic flow is toric integrable.
If $n \neq 3$ is odd, or if $\pi_1(Q)$ is infinite,
then the cosphere bundle of $Q$ is equivariantly contactomorphic 
to $\T^n \times S^{n-1}$, the cosphere bundle of $\T^n$.
\end{maintheorem}

\begin{proof}
To begin, we may assume that $n \neq 3$ because otherwise
the claim follows immediately from  \cite[Theorem 1.1]{cl}; 
(see footnote \ref{Chris}).
Similarly, we may assume that $n >  1$.

Since the geodesic flow on $Q$ is toric integrable, the cosphere bundle $S(T^*Q)$ is a 
$(2n -1)$-dimensional contact toric manifold. 

Assume first that $\pi_1(Q)$ is infinite.
By the homotopy long exact sequence for the bundle $S^{n-1} \to S(T^*Q) \to Q$,
this implies that $\pi_1(S(T^*(Q))  = \pi_1(Q)$ is infinite as well.
Hence, if $n > 3$, then
by the classification of contact toric manifolds described
in Theorem~\ref{contactthm},
$S(T^*{Q})$ is 
equivariantly contactomorphic  to $\T^k \times S^{2n-k-1}$ for some $0<k\leq{n}$. 
Finally, by
Proposition \ref{productcosphere}, 
this is impossible unless $k = n$.
Similarly, if $n = 2$, then by the classification of contact toric $3$-manifolds described
in parts (i) and (ii) of  \cite[Theorem 2.18]{l2},
$S(T^*Q)$ is homeomorphic to either $\T^3$ or $S^1 \times S^2$.
By Proposition \ref{productcosphere}  and  \ref{Qisatorus}, this is only possible if 
$Q$ is homeomorphic to $\T^2$. 
Hence, the claim follows immediately from 
\cite[Theorem 1.3]{l2}.

So assume instead that $\pi_1(Q)$ is 
finite and that $n > 3$ is odd.
If $\tQ$ is the 
universal 
cover of $Q$, 
then $\pi_1\big(\tQ\big)$ and $\pi_1\big(S\big(T^*\tQ\big) \big)$ are
trivial.
Moreover,  $S\big(T^*{\tQ}\big)$ is a 
finite-sheeted 
cover of $S(T^*Q)$,
and so  Lemma \ref{coveringcontact} implies that
$S\big(T^*\tQ\big)$ is a contact toric manifold. 
Therefore, by Theorem~\ref{contactthm}, 
$S\big(T^*{\tQ}\big)$  
can be given the structure of a principal circle bundle 
with curvature $\omega$ over a symplectic toric orbifold  $(N,\omega)$.
Since $\tQ$ is oriented,
this contradicts Proposition \ref{Reebtypecosphere}.
\end{proof}

\begin{coro}\label{maincor}
Let $Q$ be a compact, connected  $n$-dimensional Riemannian manifold,  
and assume that the geodesic flow is toric integrable.
If $n \neq 3$ is odd,  or if $\pi_1(Q)$ is infinite, then
$Q$ is homeomorphic to  $\T^n$.
If $n \leq 3$ and $\pi_1(Q)$ is finite, then $Q$ is either diffeomorphic
to the sphere $S^2$, the projective plane $\R P^2$, or the 
lens space $S^3/\Z_\ell$ for some $\ell \geq 1$.
\end{coro}

\begin{proof}
If $n \neq 3$ is odd, or if
$\pi_1(Q)$ is infinite, then
Theorem \ref{mainresult} implies that the  cosphere bundle of $Q$ is homeomorphic to $\T^n \times S^{n-1}$. 
By Proposition \ref{Qisatorus}, this is only possible if 
$Q$ is homeomorphic to $\T^n$. 

So assume that 
$\pi_1(Q)$ is finite.
If $n = 3$, then by
Theorem 1.1 in \cite{cl}, 
the fundamental group of
the cosphere bundle $S(T^*Q)$ is  the cyclic group $\Z_\ell$, for some $\ell \geq 1$. 
By Proposition~\ref{Qislens},
this implies that $Q$ is a lens space.
If $n = 2$, the claim is obvious.
\end{proof}

\begin{remark}
Let $Q$ be a compact $n$-dimensional Riemannian manifold with 
$n > 2$ even and $\pi_1(Q)$ finite.
It is easy to check that 
Theorem~\ref{contactthm} and Proposition \ref{Reebtypecosphere} imply
that if the odd Betti numbers
$\beta_{2i+1}(Q) := \dim\big( H^{2i+1}(Q;\R)\big) $ do not all vanish,
then the geodesic flow is not toric integrable.
However, in many cases the cohomological techniques in this paper 
are not sufficient to  prove that the cosphere bundle $S(T^*Q)$
is not a contact toric manifold.
For example, let $Q = S^2 \times S^2$ and let $M$ be the principal circle bundle
over $S^2 \times S^2 \times S^2$ with first Chern class
$a + b + 2c$, where $\{a,b,c\}$ is the natural
basis for  $H^2(S^2 \times S^2 \times S^2;\Z) = \Z^3$.
Then $M$ is a contact toric manifold and
$$ 
H^i(S(T^*Q;\Z))= 
H^i(M;\Z) =  
\begin{cases}
\Z & i = 0 \mbox{ or } 7  \\
\Z^2 & i = 2 \mbox{ or } 5 \\
\Z_4 & i = 4 \\
0 & \mbox{otherwise}.
\end{cases}
$$
In fact, we don't know if $S^2 \times S^2$ admits a toric
integrable geodesic flow.
A fortiori,
we don't know whether 
the list described in Corollary~\ref{maincor}
includes {\em every} compact,  connected Riemannian manifold with toric 
integrable geodesic flow. 
\end{remark}

\begin{remark}
Let $Q$ be a compact, connected  $n$-dimensional Riemannian manifold,  
and assume that  the geodesic flow is toric integrable. 
Even if   $n \neq 3$ is odd or $\pi_1(Q)$ is infinite,
we don't know whether $Q$ is necessarily {\em diffeomorphic} to  $\T^n$.
This doesn't follow from Corollary~\ref{maincor};
for example, 
when $n \geq 5$ there are fake tori that are homeomorphic to 
$\T^n$, but not diffeomorphic (see \cite{ks}).
However, Theorem~\ref{mainresult} would imply that $Q$ is diffeomorphic to $\T^n$  if we could
{\em also} prove the following:
\begin{quote}
If $Q$ and $Q'$ are homeomorphic compact manifolds and
$S(T^*Q)$ is contactomorphic to $S(T^*Q')$, then  $Q$ is diffeomorphic
to $Q'$.
\end{quote}
However, this statement seems very difficult to prove; it is closely
related to many important problems in symplectic topology; cf. \cite{Ab}
\end{remark}

\section{Contact toric manifolds}

In this section, we 
(partially) calculate the cohomology of all  
compact
contact toric manifolds with $\dim(M) > 5$.
This calculation relies heavily on Lerman's classification
of 
compact
contact toric manifolds \cite{l2}, which builds on ideas
of  Banyaga and Molino, and  Boyer and Galicki  \cite{bm, bg}.

\begin{theorem}\label{contactthm}
Let $M$ be a compact, connected, contact toric manifold with $\dim(M)=2n-1>5$.  
\begin{itemize}
\item[(i)] If $\pi_1(M)$ is infinite, then
$M$ is 
equivariantly contactomorphic 
to $\T^k \times S^{2n-k-1}$ 
for some
$0 < k \leq n$. 
\item[(ii)] If $\pi_1(M)$ is finite, then  $M$ can be given the structure of a principal 
circle bundle 
with curvature $\omega$ over a symplectic toric orbifold  $(N,\omega)$.
\end{itemize}
Here the contact toric structure on
$\T^k \times S^{2n-1-k}$ is induced by the restriction of  
\begin{gather*}
\sum p_i d q_i +  \textstyle\frac{1}{2} \sqrt{-1} \sum \left( 
\overline{z}_i d z_i - z_i d \overline{z}_i \right) 
\ \ \in \Omega^1\big( T^k \times R^k \times \C^{n-k}) \big) \quad \mbox{to} \\
\big\{(q,p,z) \in \T^k \times \R^k \times \C^{n-k} \  \big| \
\|p\|^2 + \|z\|^4 = 1 \big\}.
\end{gather*}
In particular, 
$\T^n \times S^{n-1}$ 
 gets its contact toric structure as
the cosphere bundle of $\T^n$.
\end{theorem}

\begin{proof}
This proof is adapted from the proof of Theorem 1.3 in \cite{l2}.

The contact structure on $M$ is induced by a contact form $\alpha$ that is
invariant under the action of $\T^{n}$ on $M$. 
For any $X \in \mathfrak{t}$, 
let $X_M$ be the corresponding vector field on $M$. The {\bf  $\mathbf{\alpha}$-moment map}
$\Psi_\alpha \colon M \to \mathfrak{t}^*$
is then defined by
$$
\langle \Psi_{\alpha}(p) , X \rangle =  \alpha_p (X_M(p) )
$$
for all $p$ in $M$ and $X \in \mathfrak{t}$.
Let $C(M) = \{t\Phi_{\alpha}(M) \ | \ t \in [0,\infty) \}$
be the cone on the image of $\Psi_{\alpha}$ in $\mathfrak{t}^*$. 

By \cite[Lemma 2.12]{l2}, $\Psi_{\alpha}(p) \neq 0$ for all $p \in M$.
Since $n >2$, this implies that $C(M)$ is a convex polyhedral cone
\cite[Theorem 1.2]{l3}.

Assume
first that there exists  $X \in \ft$ such that $\langle \Psi_{\alpha}(p) ,X \rangle > 0$
for all $p \in M$.  
Theorem 1.1 
in \cite{l1} asserts that every contact toric manifold of 
this type has finite fundamental group.  
Moreover, by \cite[Theorem 4.3]{ls} (see also \cite{bg}), every contact toric manifold of this type 
can be given the structure of a principal 
circle bundle over a symplectic toric orbifold $(N,\omega)$.
In fact, this circle bundle has curvature $\omega$.
To see this, note that in 
their proof they show that
we may choose $X$ and $\alpha$ so that the vector field $X$ generates
the circle action and also so that 
$\langle \Psi_{\alpha}(p), X \rangle = 1$.
Hence, $\alpha$ is a connection $1$-form on the bundle 
$S^1 \to M \stackrel{\pi}{\to} N$. Finally,
it is easy to check that $\pi^*(\omega) = d \alpha$.

Next, assume that the $\T^n$ action is free.
Since $n > 3$, by \cite[Theorem C]{bm} (alternatively, by \cite[Theorem 2.18 (3)]{l2})
$M$ is equivariantly contactomorphic to $\T^n \times S^{n-1}$,
the cosphere bundle of $\T^n$.
Clearly, since $n > 2$, $\pi_1(T^n \times S^{n-1}) = \Z^n$.

Finally, assume that the action is not free and that
there does not exist any $X \in \ft$ such that $\langle \Psi_{\alpha}(p), X \rangle > 0$
for all $p \in M$. Lemma 4.5 in \cite{l2} implies that, since  the action is not free, $C(M) \neq \ft^*$.
Since $C(M)$ is a closed 
convex cone,
this implies that the maximal linear
subspace of $C(M)$ has dimension $k$, where $0 < k < n$; (see Remark \ref{closedcone}).
Hence, 
\cite[Theorem 2.18(4)]{l2}
implies that $C(M)$ is isomorphic to the moment cone of 
$M' = \T^k \times S^{2n-1-k}$
and that $M$ is equivariantly 
contactomorphic to $M'$.
Finally, note that the fact that $0 < k < n$
implies that $\pi_1(\T^k \times S^{2n-1-k} ) = \Z^k$ is infinite.

Since these are the only three possibilities, the claim follows immediately. \end{proof}

\begin{remark}\label{closedcone}
Let $C \subset \ft^*$ be a closed convex cone.
If there does not exist $X \in \ft$ such that $\langle c, X \rangle>  0$
for all $c \in C \smallsetminus \{0\}$, then
$C$ contains a non-trivial linear subspace.
This follows by an easy inductive argument
from the well-known fact that if $C \neq \ft^*$,
there exists $\xi \in \ft$ such that $\langle c, \xi \rangle \geq 0$
for all $c \in C$.
\end{remark}

Given a space $X$,
let $\beta_j(X) := \dim H^j(X;\R)$
be the $j$th Betti number of $X$.
The cohomology
of the manifolds listed in part (i) of  Theorem \ref{contactthm} are 
given immediately by the K\"{u}nneth formula. 

\begin{prop}\label{productBetti}
Given  integers $k$ and $n$ with $0<k\leq n$,  
\begin{equation*}
\beta_{m}\big(\T^k \times S^{2n-k-1}\big) = 
\sum_{p=0}^{m}\beta_{p}\big(\T^k\big)\beta_{m-p}\big(S^{2n-k-1}\big) 
= \binom{k}{m}+\binom{k}{2n-m-1} \qquad 
\mbox{for all} \ m \in \Z,
\end{equation*}
where by convention $\binom{k}{m}=0$ unless $0 \leq m \leq k$.
Moreover, $H^*\big(\T^k \times S^{2n - k -1}; \Z \big)$ is torsion-free.
\end{prop}

In order to  analyze
the Betti numbers of the 
the second class of manifolds  described in Theorem \ref{contactthm},
we need to review some facts about symplectic toric orbifolds. 
Since they are K\"{a}hler, the Hard Lefschetz Theorem implies the following.

\begin{theorem}\label{hardlef}
Let $(N^{2n-2},\omega)$ be a compact symplectic toric orbifold. Then, the map
\begin{displaymath}
\wedge{[\omega]}:H^{i}(N;\R) \to H^{i+2}(N;\R)
\end{displaymath}
is injective for all $i <  n-1$. 
\end{theorem} 

Our next result is a consequence of a theorem of Danilov 
\cite{d},
which describes the cohomology ring of  symplectic toric orbifolds.

\begin{prop}\label{symptoricorbcoho}
Let $(N,\omega)$ be a compact, connected
symplectic toric orbifold. As a ring, $H^{*}(N;\R)$ is generated by $H^{2}(N;\R)$; in particular, $H^{i}(N;\R)=0$ for $i$ odd. 
\end{prop}

\begin{prop}\label{ReebtypeBetti}
Let $M$ be a principal $S^1$-bundle 
with curvature $\omega$ over a compact symplectic toric orbifold  $(N,\omega)$.
Then, 
$$\beta_{i}(M) =
\begin{cases}
0 & \mbox{if $i < n$ is odd, } \\
\beta_i(N) - \beta_{i-2}(N) & \mbox{if  $i < n + 1$ is even}.
\end{cases}
$$
\end{prop}

\begin{proof}
We may assume without loss of generality that $N$ is connected.
Consider the Gysin sequence
\begin{equation}\label{GysinM}
\cdots \to H^{i-2}(N;\R) \overset{\wedge{[\omega]}}\longrightarrow  
H^{i}(N;\R) \to H^{i}(M;\R) \to H^{i-1}(N;\R) \overset{\wedge{[\omega]}}\longrightarrow  H^{i+1}(N;\R) \to \cdots. 
\end{equation}
If $i$ is odd then $H^i(N;\R)$ = 0 by Proposition~
\ref{symptoricorbcoho},
and if $i < n$ then $\wedge [\omega] \colon H^{i-1}(N;\R) \to H^{i+1}(N;\R)$ is injective by Theorem~\ref{hardlef}.
Hence, the first case follows immediately from  \eqref{GysinM}.
Similarly, if $i$ is even then $H^{i-1}(N;\R) = 0$,
and if $i < n+1$ then
$\wedge [\omega] \colon H^{i-2}(N;\R) \to H^{i}(N;\R)$ is injective.
Hence, the second case also follows immediately from  \eqref{GysinM}.
\end{proof}

Our final  result is that contact 
toric manifolds behave nicely with respect to finite-sheeted covers.

\begin{lemma}\label{coveringcontact}
Let $M$ be a compact, contact toric manifold. If 
$\widetilde{M}$ is any finite-sheeted cover
of $M$, then $\widetilde{M}$ is a compact, contact toric manifold.
\end{lemma}

\begin{proof}
Let $p:\widetilde{M} \to M$ be the covering map, and $\alpha$ an invariant
 contact form on $M$ such 
that $\ker(\alpha)$ is the contact structure on $M$. Then, since $p$ is a local diffeomorphism, 
$p^*\alpha \wedge d({p}^*\alpha)^{n-1} = p^*(\alpha \wedge (d\alpha)^{n-1})$ is  nowhere zero,
and so $p^*\alpha$ is a contact form on $\widetilde{M}$. 
Since the covering is finite, 
the torus action on $M$ lifts to a torus action 
on $\widetilde{M}$.
This lifted action is effective and preserves $p^*\alpha$. 
\end{proof}

\section{Cosphere bundles}

In this section, we 
(partially) calculate the cohomology
of cosphere bundles, following \cite{mmo}.
We begin with an easy consequence of the Gysin sequence.

\begin{lemma}\label{Gysin}
Let $M = S (T^* Q)$, where $Q$ is a connected $n$-dimensional manifold.
If $Q$ is orientable, fix an orientation and
let  $e \in H^n(Q;\Z)$ be the Euler class of the bundle
$S^{n-1} \to M \overset{p} \to Q$. 
Then
\begin{itemize}
\item [(a)] 
$H^i(M;\Z) = H^i(Q;\Z)$ for all  $i < n-1$, and
\item [(b)]
$
H^{n-1}(M;\Z) =
\begin{cases}   
H^{n-1}(Q;\Z) \oplus \Z  & \text{if } Q \mbox{ is orientable and }  e = 0, \\
H^{n-1}(Q;\Z)  & \text{otherwise.}
\end{cases}
$
\end{itemize}
\end{lemma}

\begin{proof}
Assume first that $Q$ is orientable.
The Gysin sequence of the bundle  
$S^{n-1} \to M \overset{p} \to Q$
is exact:
\begin{equation}
\label{GysinQ}
\cdots\to H^{i-n}(Q;\Z) 
\overset{\cup{e}}\longrightarrow H^{i}(Q;\Z)
\to H^{i}(M;\Z) 
\to H^{i-n+1}(Q;\Z)
\overset{\cup{e}}\longrightarrow H^{i+1}(Q;\Z)
\to\cdots. 
\end{equation}
If $i < n-1$, then $H^{i-n}(Q;\Z)=H^{i-n+1}(Q;\Z)=0$. 
Hence, \eqref{GysinQ} implies claim (a).
If $Q$ is compact, then $H^n(Q;\Z) = \Z$; if $Q$ is not compact,
then $H^n(Q;\Z) = 0$.
In either case, $H^0(Q;\Z) = \Z$ and 
so if $e$ is not zero, then 
$\cup e \colon H^0(Q;\Z) \to H^n(Q;\Z)$ is injective.
Hence,  since  $H^{(n-1) - n}(Q;\Z) = 0$ and
$H^{(n-1) - n + 1}(Q;\Z) = \Z$,  claim (b) follows from \eqref{GysinQ} with $i = n-1$. 

So assume instead that $Q$ is not orientable.
In this case, the Gysin sequence with twisted integer coefficients is exact:
\begin{equation}\label{twistedGysinQ}
\cdots\to H^{i-n}(Q;\mathcal{Z}) 
\to
H^{i}(Q;\Z)
\to H^{i}(M;\Z) 
\to H^{i-n+1}(Q;\mathcal{Z})
\to
H^{i+1}(Q;\Z)
\to\cdots. 
\end{equation}
For any $i \leq n-1$, $H^{i-n}(Q;\mathcal{Z}) = H^{i-n+1}(Q;\mathcal{Z}) = 0$.
Hence, \eqref{twistedGysinQ} implies (a) and (b).
\end{proof}

Our main result is adapted from the statements of \cite[Theorems  2.1,  2.3,
and  3.1]{mmo}, and the proof of \cite[Theorem 2.3]{mmo}.

\begin{theorem}[McCord, Meyer, and Offin]\label{cosphereBetti}
Let $M = S(T^*(Q))$,
where $Q$ is a compact, connected
$n$-dimensional manifold. Then  
\begin{enumerate}
\item [(i)] if $Q$ is orientable, 
then $\beta_{i}(M) = \beta_{n-i}(M)$ for all 
$i \in \{2,\dots,n-2\}$; 
\item [(ii)] if $Q$ is orientable and $n > 2$, then
$$
\beta_{n-1}(M) =
\begin{cases}   
\beta_1(M) + 1  & \text{if \  
$\sum_{i=0}^{n-2} (-1)^i\beta_i(M) + (-1)^{n-1} \beta_1(M) + (-1)^n = 0$,} \\
\beta_1(M)  & \text{otherwise;}
\end{cases}
$$
\item [(iii)] if $Q$ is orientable and $n=2$, 
then $\beta_{1}(M) \neq 1$;  
and
\item [(iv)] if  $Q$ is not orientable, 
then $H^n(M;\Z)$  
is not torsion-free.
\end{enumerate} 
\end{theorem}

\begin{proof}
Assume first that $Q$ is orientable;
fix an orientation on $Q$.
Since $Q$ is compact, the 
Euler class of the bundle $S^{n-1} \to M \overset{p}\to Q$ is zero 
exactly if  the Euler characteristic $\chi(Q) = \sum_{i=0}^n (-1)^i\beta_i(Q)$ is zero.
Therefore, Lemma~\ref{Gysin} implies that
\begin{gather}\label{QMi}
\beta_i(M) = \beta_i(Q) \quad 
\mbox{for all}  \ i < n-1, \quad \mbox{and} \\
\label{QMn-1}
\beta_{n-1}(M) =
\begin{cases}   
\beta_{n-1}(Q) + 1  & \text{if 
$\chi(Q) = 0$,} \\
\beta_{n-1}(Q)  & \text{otherwise.}
\end{cases}
\end{gather}
Moreover, by Poincar\'e duality
\begin{equation}\label{PD}
\beta_i(Q) = \beta_{n-i}(Q) \quad 
\mbox{for all} \ i.
\end{equation}
Equations  \eqref{QMi} and \eqref{PD} imply that claim (i) 
holds,  and also that $\beta_{n-1}(Q) = \beta_1(M)$ if $n > 2$.
Hence, since $\beta_n(Q) = 1$, claim (ii) follows
from \eqref{QMn-1}.
Finally, if $n = 2$, then  $Q$ is a compact orientable
surface of genus $g$; 
in particular 
$H^1(Q;\R) = \R^{2g}$.
If $g \neq 1$, then $\chi(Q) \neq 0$ and so \eqref{QMn-1} implies that
$\beta_1(M) = \beta_1(Q) = 2g \neq 1$.  On the other hand,
if $g = 1$, then $\chi(Q) = 0$ and so \eqref{QMn-1} implies that
$\beta_1(M) = \beta_1(Q) + 1 = 3 \neq 1$.
This completes the proof of claim (iii).

Now assume that $Q$ is not 
orientable; the Gysin sequence with twisted integer coefficients is exact:
\begin{displaymath}
\cdots  \to H^{0}(Q;\mathcal{Z}) \to H^n(Q;\Z) \to H^n(M;\Z) 
\to H^{1}(Q;\mathcal{Z}) \to \cdots.
\end{displaymath}
Claim (iv) follows immediately from the
facts that  $H^n(Q;\Z) = \Z/2\Z$ and $H^0(Q;\mathcal{Z})=0$.
\end{proof}

\section{Cosphere bundles which are contact toric manifolds}

In this section, we compare 
our earlier calculations of 
the cohomology
of cosphere bundles and the cohomology
of the contact toric manifolds listed
Theorem \ref{contactthm}. We use this
to show that $\T^n \times S^{n-1}$
is the only one of these manifolds that 
can  be given the structure of the cosphere 
bundle of an odd-dimensional, oriented manifold.
Moreover, we show that if the cosphere bundle of a manifold $Q$ is
$\T^n \times S^{n-1}$, then $Q$ is homeomorphic to $\T^n$.

\begin{prop}\label{productcosphere}
If the cosphere bundle $M = S(T^*Q)$ of a manifold $Q$
is homeomorphic to $\T^k \times S^{2n-k-1}$  
for some integers $k$ and $n$ with $0<k\leq n$,  
then  $k=n$. 
\end{prop}

\begin{proof} 
Assume on the contrary that $0 < k < n$.
Proposition~\ref{productBetti} 
implies that
\begin{gather*}
\beta_{1}\big(\T^k \times S^{2n-k-1}\big) = k \quad \mbox{and} \\
\beta_{n-1}\big(\T^k \times S^{2n-k-1}\big) =
\left\{ \begin{array}{ll} 1 & \textrm{if } k=n-1, \\ 0 & \textrm{if } k < n-1. \end{array} \right.
\end{gather*}
In particular $\beta_{1}\big(S^1 \times S^2\big)=1$,  
and so  $S^1 \times S^2$ fails to be a cosphere bundle 
of an oriented manifold
by Theorem \ref{cosphereBetti}(iii). 
Therefore, we may assume that $n >2$.
Hence, 
Theorem  
\ref{cosphereBetti}(ii) 
implies that
if $\T^{k} \times S^{2n-k-1}$ is the cosphere bundle of an orientable manifold, 
then $\beta_{n-1}\big(\T^{k} \times S^{2n-k-1}\big) \geq \beta_{1}\big(\T^{k} \times S^{2n-k-1}\big)$. 
But this contradicts the equations above.
On the other hand, 
Proposition~\ref{productBetti} and Theorem~\ref{cosphereBetti}(iv) imply
that $T^k \times S^{2n-k-1}$ is
not the cosphere bundle of a manifold which is not orientable.
\end{proof}

\begin{prop}\label{Reebtypecosphere}
Assume that  $M$ is a principal $S^1$-bundle 
with curvature $\omega$ over a  symplectic toric orbifold  $(N,\omega)$,
{\em and} that $M$ is the cosphere bundle of a compact $n$-dimensional manifold $Q$
with $n > 3$. If $n$ is odd, then $Q$ is not orientable;
if $n$ is even, then  $\beta_{2i+1}(Q) = 0$ for all $i$.
\end{prop}

\begin{proof}
We may assume without loss of generality that $Q$ is 
connected; in particular,
\begin{equation}\label{betti0}
\beta_0(M) = 1.
\end{equation}
Moreover, by
Proposition \ref{ReebtypeBetti}
\begin{equation}\label{oddbetti}
\beta_{i}(M)=0 \quad \mbox{for all } i<n \mbox{ odd}. 
\end{equation}

Assume first that $n$ is odd and $Q$ is orientable.
Then \eqref{oddbetti} and
Theorem \ref{cosphereBetti}(i) 
imply
that 
\begin{equation}\label{bettii}
\beta_i(M) = 0 \quad 
\mbox{for all} \ i \in \{1,\dots,n-2\}.
\end{equation}
Applying Theorem \ref{cosphereBetti}(ii)
to \eqref{betti0} and \eqref{bettii}, we see that 
\begin{equation}\label{bettin-1}
\beta_{n-1}(M) = 1.
\end{equation}
By induction and Proposition \ref{ReebtypeBetti},
\begin{equation}\label{induct}
\beta_i(N) = \sum_{j=0}^i \beta_j(M) \quad 
\mbox{ for $i <  n + 1$ even}.
\end{equation}
Since $n - 1 > 2$, \eqref{betti0}, \eqref{bettii},
\eqref{bettin-1}, and \eqref{induct} together
imply that $\beta_2(N) = 1$ and $\beta_{n-1}(N) = 2$.
But this is impossible,  since  as a ring
$H^2(N;\R)$ generates $H^*(N;\R)$ by Theorem 
\ref{symptoricorbcoho}.

So assume instead that $n$ is even.
By Lemma~\ref{Gysin}(a),  \eqref{oddbetti} implies that  $\beta_i(Q) = 0$
for all $i < n-1$ odd.
If $Q$ is orientable,  then  
$\beta_{n-1}(Q) = \beta_1(Q)$ by Poincar\'e duality.
Moreover, $\beta_1(Q) = 0$ since $1 < n-1$.
On the other hand, if $Q$ is not orientable, then
$\beta_{n-1}(Q) = 0$ by \eqref{oddbetti} and Lemma~\ref{Gysin}(b).
\end{proof}

\begin{prop}\label{Qisatorus}
If the cosphere bundle $M = S(T^*Q)$ of a manifold $Q$
is homeomorphic to $\T^n \times S^{n-1}$, then $Q$ is homeomorphic to $\T^n$.
\end{prop}

\begin{proof}
Note first that 
Proposition~\ref{productBetti} and Theorem~\ref{cosphereBetti}(iv)  imply
that $Q$ is  orientable;
fix an orientation on $Q$.
If $n = 2$, the claim follows from the fact that the torus is the only
closed surface
with  $\beta_1(S(T^*Q)) = \beta_1(T^2 \times S^1) = 3.$
So assume that $n > 2$.
If $\tQ$ is the universal cover of $Q$, then 
$S\big(T^*\tQ\big)$ is 
homeomorphic
to $\R^n \times S^{n-1}$. 
In particular,
\begin{equation}\label{betticov}
H^i \big(S\big(T^*\tQ\big);\Z\big) =
\begin{cases}   
\Z & \text{if $i = 0$ or $n-1$,} \\
0  & \text{otherwise.}
\end{cases}
\end{equation}

Since $S\big(T^*\tQ\big)$ is not compact, $\tQ$ is not compact.
Therefore, $H^n\big(\tQ;\Z\big) = 0$; {\em a fortiori},
the Euler class of the bundle $S^{n-1} \to S\big(T^*\tQ\big)  \overset{p}\to \tQ$ vanishes.
By Lemma~\ref{Gysin} and \eqref{betticov}, this implies that
$H^i\big(\tQ;\Z\big) = 0$ for all  $i > 0$.
Since $\tQ$ is simply connected,
by the Hurewicz Theorem this implies that
$\pi_i\big(\tQ\big) = 0$ for all $ i > 0.$
By the homotopy long exact sequence for
the covering $\tQ \to Q$,
this implies that 
$$\pi_i (Q) = 0\quad 
\mbox{for all} \  i > 1.$$
On the other hand,
by the homotopy long exact 
sequence for the bundle $S^{n-1} \to S(T^*Q) \to Q$, we have 
$$\pi_1(Q)= \pi_1( S (T^*Q)) = \pi_1(\T^n \times S^{n-1})=\Z^n.$$ 
Thus, $Q$
is an Eilenberg-MacLane space of type 
$K(\Z^n,1)$. 
Since Eilenberg-MacLane spaces are 
unique up to weak homotopy equivalence, it follows that $Q$  
and $\T^n$ are weakly homotopy equivalent. 
Since $Q$ is a manifold, and hence a CW complex, this
implies that $Q$  and $\T^n$ are  homotopy equivalent 
by Whitehead's theorem.

Moreover, in \cite{hw}, Hsiang and Wall 
prove that if a closed $n$-manifold with $n\geq 5$ is homotopy equivalent
to the $n$-torus, it is homeomorphic to the $n$-torus. 
The same conclusion holds in dimension four and follows from M.H. Freedman's 
celebrated work in \cite{f}. 
Finally, assume that $n = 3$.
Since $Q$ is a closed manifold and $\pi_1(Q) = \Z^3$ cannot be
written as a non-trivial free product, Perelman's proof of the
Poincar\'e conjecture implies that $Q$ is prime \cite{p1,p2,mt1}.
Since $\pi_1(Q) \neq \Z$, this implies that $Q$ is irreducible.
Moreover, since $H_1(Q,\R) \neq 0$, $Q$
contains an orientable, incompressible surface.
Hence, since $Q$ is oriented, Waldhausen's theorem
implies that $Q$ is homeomorphic to $T^3$
\cite{h}.
\end{proof}

\begin{prop}\label{Qislens}
If $M$ is the cosphere bundle of a compact $3$-manifold $Q$
and $\pi_1(M) = \Z_\ell$ for some $\ell \geq 1$,
then $Q$ is diffeomorphic  
to the lens space $S^3/\Z_\ell$.
\end{prop}

\begin{proof}
By the long exact sequence for the fibration $S^{2} \to S(T^*Q) \to Q$, 
the assumptions imply
that $\pi_1(Q) = \Z_\ell$.
Hence, the claim follows trivially from Thurston's elliptizaton conjecture,
which states that every closed three-manifold with finite fundamental group
is 
quotient of $S^3$ by a finite subgroup of $SO(4)$ acting freely.
The elliptization conjecture is itself an easy consequence  
of
Thurston's geometrization conjecture, proved by Perelman \cite{p1,p2,mt2}.
To see this, let $M$ be a prime three-manifold with finite
fundamental group.  Since $\pi_1(M)$ is finite, $M$ does not
admit any incompressible tori or Klein bottles. Therefore, the geometrization
conjecture implies that $M$ admits a locally homogeneous metric.
Finally, $S^3$ is the only compact model geometry.
\end{proof}

%\bibliographystyle{plain}
%\bibliography{biblio}

\end{document}